\documentclass{amsart}
\usepackage{amsmath}
\usepackage{amssymb}
\usepackage{amsthm}

\usepackage[colorlinks]{hyperref}

\newtheorem{theorem}{Theorem}[section]
\newtheorem{lemma}[theorem]{Lemma}
\newtheorem{proposition}[theorem]{Proposition}

\theoremstyle{definition}

\theoremstyle{remark}
\newtheorem{remark}[theorem]{Remark}

\numberwithin{equation}{section}

\DeclareMathOperator{\Lip}{\mathrm{Lip}}
\DeclareMathOperator{\cco}{\overline{co}}
\DeclareMathOperator{\spt}{spt}
\DeclareMathOperator*{\esssup}{ess~sup}

\begin{document}

\title[Minimax formula for quasiconvex Hamiltonians]{Two approaches to minimax formula of the additive eigenvalue for quasiconvex Hamiltonians}

\author{Atsushi~Nakayasu}
\address{
Graduate School of Mathematical Sciences, The University of Tokyo
3-8-1 Komaba, Meguro, Tokyo, 153-8914 Japan
}
\curraddr{}
\email{ankys@ms.u-tokyo.ac.jp}
\thanks{
The work of the author was supported by a Grant-in-Aid for JSPS Fellows No.\ 25-7077 and the Program for Leading Graduate Schools, MEXT, Japan.
}

\subjclass[2010]{Primary 35F21; Secondary 49L25, 26B25, 26B05}

\keywords{
Additive eigenvalue problem,
Hamilton-Jacobi equations,
Minimax formula,
Quasiconvex Hamiltonians,
Quasiconvex functions,
Lipschitz continuous functions
}

\date{}

\begin{abstract}

Two different proofs for an inf-sup type representation formula (minimax formula) of the additive eigenvalues corresponding to first-order Hamilton-Jacobi equations are given for quasiconvex (level-set convex) Hamiltonians not necessarily convex.
The first proof, which is similar to known proofs for convex Hamiltonians, invokes a Jensen-like inequality for quasiconvex functions instead of the standard Jensen's inequality.
The second proof is completely different with elementary calculations.
It is based on convergence of derivatives of mollified Lipschitz continuous functions whose proof is also given.
These methods also relate to an approximation problem of viscosity solutions.
\end{abstract}

\maketitle

\section{Introduction}
\label{s:intro}

It is well-known that the additive eigenvalue for a Hamilton-Jacobi equation has an inf-sup type representation formula if the Hamiltonian is continuous, convex and coercive.
In this article we will introduce two approaches to this problem.
One is similar to known arguments using Jensen's inequality directly to the Hamiltonian
while the other one invokes Clarke's generalized gradient.
Both of these two approaches will derive the representation formula under a weaker assumption on the Hamiltonian.
We now stress that the latter approach is rather new as far as the author knows
and using a crucial lemma on convergence of mollifications of Lipschitz continuous functions (Lemma \ref{t:mollip}),
whose proof will be given in Section \ref{s:pfmollip}.

For simplicity, we consider first-order Hamilton-Jacobi equations in the periodic setting of the form
\begin{equation}
\label{e:aep}
H(x, D u) = a \quad \text{in $\mathbf{T}^N := \mathbf{R}^N/\mathbf{Z}^N$}
\end{equation}
with a parameter $a \in \mathbf{R}$.
Here, $H = H(x, p) \colon \mathbf{T}^N\times\mathbf{R}^N \to \mathbf{R}$ is a function called a Hamiltonian satisfying the following conditions:
\begin{itemize}
\item[(A1)]
(Continuity)
$H$ is continuous on $\mathbf{T}^N\times\mathbf{R}^N$.
\item[(A2)]
(Convexity)
$H$ is convex in the variable $p \in \mathbf{R}^N$ for each $x \in \mathbf{T}^N$.
\end{itemize}
An (additive) eigenvalue is a unique constant $a \in \mathbf{R}$ such that \eqref{e:aep} admits a viscosity solution $u$ (\cite{CL83}) with Lipschitz continuity;
$D u$ denotes a gradient of the unknown function $u = u(x)$.
Then, the eigenvalue $a = c$, if exists, will satisfy the representation formulas
\begin{align}
\label{e:cvformula1}
c &= \inf_{u \in C^1(\mathbf{T}^N)}\sup_{\nabla u}H, \\
\label{e:cvformula2}
c &= \inf_{u \in \Lip(\mathbf{T}^N)}\sup_{\nabla u}H,
\end{align}
where $\nabla u$ is the graph of the classical gradients (also denoted by $\nabla u$) of $u$,
i.e.\
$$
\nabla u := \{ (x, p) \in \mathbf{T}^N\times\mathbf{R}^N \mid \text{$u(y) = u(x)+p\cdot(y-x)+o(|y-x|)$ as $y \to x$} \}.
$$
Note that Lipschitz continuous functions $u \in \Lip(\mathbf{T}^N)$ are differentiable almost everywhere by Rademacher's theorem and $\nabla u \in L^\infty(\mathbf{T}^N)$.

This kind of expression (the right-hand side of \eqref{e:cvformula1}) was found as a variational formula of Ma\~{n}\'{e}'s critical value with respect to the corresponding Lagrangian by Contreras-Iturriaga-Paternain-Paternain \cite{CIPP98}.
On the other hand, this is a pure partial differential equations problem.
In view of this, the above minimax formula was established by Fathi in the context of weak KAM (Kolmogorov-Arnold-Moser) theory powered by the viscosity solution theory;
see \cite[Section 6]{E03}.
We remark that the additive eigenvalue problem \eqref{e:aep} also appears in solving homogenization problems \cite{LPV87} and long time behaviors \cite{NR99}.
It is also known that the minimax formula is useful for computing the additive eigenvalue numerically \cite{GO04}.
This work will provide natural extensions for these theories.

In this article we extend the representation formula for general quasiconvex Hamiltonians (see (A2') below) instead of the convexity assumption (A2) with two different proofs.
\begin{itemize}
\item[(A2')]
(Quasiconvexity)
$H$ is quasiconvex in the variable $p \in \mathbf{R}^N$ for each $x \in \mathbf{T}^N$,
i.e.\
$H(x, \theta p + (1 - \theta)q) \le \max\{ H(x, p), H(x, q) \}$
for all $p, q \in \mathbf{R}^N$, $0 \le \theta \le 1$ and $x \in \mathbf{T}^N$.
\end{itemize}
We remark that the quasiconvexity is sometimes called level-set convexity since (A2') is equivalent to the condition that the sublevel sets $\{ p \in \mathbf{R}^N \mid H(x, p) \le a \}$ are convex for all $a \in \mathbf{R}$ and $x \in \mathbf{T}^N$.

Recently several authors study homogenization problems with quasiconvex Hamiltonians;
see \cite{FS05} and \cite{AS13}.
In fact, the authors of \cite{AS13} mention some relation between the eigenvalue and the minimax expression
and \cite[Proposition 6.2]{AS13} will immediately show one of the representations \eqref{e:cvformula2}.
Indeed, we can show \eqref{e:cvformula2} easily in view of Propositions \ref{t:cvest} and \ref{t:cvest2}.
On the other hand, to show \eqref{e:cvformula1} need more advanced calculations such as Lemmas \ref{t:measqconv} and \ref{t:mollip} below,
and there seem to be no results on it as far as the author knows.
We also point out that the authors of \cite{YGR08} posed a Hamiltonian of the form
$$
H(x, p) = H^2_\varepsilon(x, p) = \sigma\left(\frac{x}{\varepsilon}\right)\frac{p}{p_s}\tanh\left(\frac{p_s}{p}\right)
$$
with a positive continuous function $\sigma$, a constant $p_s > 0$ and a parameter $\varepsilon > 0$.
This Hamiltonian is quasiconvex (A2') as well as non-coercive.
Long time behavior and homogenization for this Hamiltonian have been studied in \cite{GLM14} and \cite{HNN14}.

In order to explain the main idea of one of the proofs,
let us review the known proof under the assumptions (A1) and (A2).
This proof is inspired by \cite{BJ90}, \cite[Proposition 2.2]{GO04} and \cite[Subsection 4.2]{MT14}.
First, it is easy to show the inequalities $\inf_{u \in \Lip}\sup_{\nabla u}H \le c \le \inf_{u \in C^1}\sup_{\nabla u}H$.
We hence claim $\inf_{u \in C^1}\sup_{\nabla u}H \le \inf_{u \in \Lip}\sup_{\nabla u}H$.
Now, for $u \in \Lip(\mathbf{T}^N)$ take mollifications $u_n := u*\eta_n \in C^\infty(\mathbf{T}^N)$ with the standard Friedrichs mollifier $\eta_n$.
Then, we observe that
$$
\begin{aligned}
H(x, \nabla u_n(x))
&= H\left(x, \int_{\mathbf{T}^N} \nabla u(x-y)\eta_n(y)d y\right) \\
&\le \int_{\mathbf{T}^N} H(x, \nabla u(x-y))\eta_n(y)d y \\
&\le \int_{\mathbf{T}^N} H(x-y, \nabla u(x-y))\eta_n(y)d y+\alpha_n
\le \sup_{\nabla u}H+\alpha_n
\end{aligned}
$$
for all $x \in \mathbf{T}^N$
and therefore we will have the desired inequality.
Here, we have invoked the convexity (A2) so that Jensen's inequality yields the first inequality;
the second equality follows from the continuity (A1) with some error term $\alpha_n > 0$ such that $\alpha_n \to 0$.

Our idea of the proof is to use another Jensen-like inequality for quasiconvex functions stated below.
\begin{lemma}[Fundamental inequality for quasiconvex functions]
\label{t:measqconv}
Let $f$ be a lower semicontinuous function defined on $\mathbf{R}^N$.
Then, $f$ is quasiconvex on $\mathbf{R}^N$
if and only if
$$
f\left(\int_\Omega X d\mu\right) \le \esssup_\Omega f\circ X
$$
for all measure spaces $(\Omega, \mu)$ with $\mu(\Omega) = 1$ and all $\mathbf{R}^N$-valued integrable functions $X$ on $\Omega$.
\end{lemma}
In view of this inequality, we can improve the proof for the representation formula.
The proof will be given in Section \ref{s:pfmeasqconv}.
We point out that a discrete version of Lemma \ref{t:measqconv} has already been studied in \cite{DP12}.

The other proof is one using the \emph{generalized gradients} of Lipschitz functions $u$ defined by
$$
\partial u := \cco_p\overline{\nabla u},
$$
i.e.\ $\partial u \subset \{ (x, p) \in \mathbf{T}^N\times\mathbf{R}^N \}$ is the closed convex hull with respect to $p$ of the closure of the classical gradients $\nabla u$.
This is nothing but Clarke's gradients; see \cite{C75} and \cite{C83}.
Also note that $\partial u$ is compact since $\nabla u \in L^\infty(\mathbf{T}^N)$.
Now, the quasiconvexity of $H$ implies that $\inf_{u \in \Lip}\sup_{\nabla u}H = \inf_{u \in \Lip}\sup_{\partial u}H$.
The remaining inequality $\inf_{u \in C^1}\sup_{\nabla u}H \le \inf_{u \in \Lip}\sup_{\partial u}H$ can be shown by a graph convergence of the standard mollifications of Lipschitz functions stated below.
The proof will be given in Section \ref{s:pfmollip}.

\begin{lemma}[Convergence of mollifications]
\label{t:mollip}
Let $u \in \Lip(\mathbf{T}^N)$ and let $u_n \in C^\infty(\mathbf{T}^N)$ be the standard mollification $u*\eta_n$.
If a sequence $(x_n, p_n) \in \nabla u_n$ converges to $(x, p) \in \mathbf{T}^N\times\mathbf{R}^N$,
then $(x, p) \in \partial u$.
\end{lemma}

In our arguments, the quasiconvexity (A2') is essential.
We point out that the authors of \cite{ATY13} and \cite{ATY14} obtain partial results on homogenization for Hamiltonians without convexity such as
$$
H(x, p) = (|p|^2-1)^2-V(x)
$$
with a bounded function $V$.
Representation formula for such Hamiltonians is an open problem.

This paper is organized as follows.
In Section \ref{s:main} we give a complete statement of our main result on the minimax formula.
We prove it in Section \ref{s:pfmeasqconv} by using the fundamental inequality for quasiconvex functions (Lemma \ref{t:measqconv})
while we give another proof in Section \ref{s:pfmollip} with the generalized gradient and Lemma \ref{t:mollip}.
The contexts of Sections \ref{s:pfmeasqconv} and \ref{s:pfmollip} are independent
so the reader can skip Section \ref{s:pfmeasqconv}.

\section{Statement of the main theorem}
\label{s:main}

In this section we give a rigorous definition of the viscosity solutions and the eigenvalues of the Hamilton-Jacobi equations \eqref{e:aep} and a complete statement of the main theorem on the minimax formula.
First, define the graphs of \emph{superdifferentials} $D^+ u$ and \emph{subdifferentials} $D^- u$ for a function $u$ by
\begin{align*}
D^+ u := \{ (x, p) \in \mathbf{T}^N\times\mathbf{R}^N \mid \text{$u(y) \le u(x)+p\cdot(y-x)+o(|y-x|)$ as $y \to x$} \}, \\
D^- u := \{ (x, p) \in \mathbf{T}^N\times\mathbf{R}^N \mid \text{$u(y) \ge u(x)+p\cdot(y-x)+o(|y-x|)$ as $y \to x$} \}.
\end{align*}
Note that the superdifferentials and the subdifferentials can be characterized by smooth functions touching $u$ from above or below;
see \cite[Section 2]{CIL92}.
A function $u \in \Lip(\mathbf{T}^N)$ is called a \emph{viscosity subsolution}, a \emph{viscosity supersolution} or a \emph{viscosity solution} of the Hamilton-Jacobi equation \eqref{e:aep} with $a \in \mathbf{R}$
if
$$
\sup_{D^+ u}H \le a,
\quad \inf_{D^- u}H \ge a,
\quad \sup_{D^+ u}H \le a \le \inf_{D^- u}H,
$$
respectively.
A \emph{subeigenvalue}, a \emph{supereigenvalue} or an \emph{eigenvalue} of the additive eigenvalue problem \eqref{e:aep} is a constant $a \in \mathbf{R}$
such that there exists at least one viscosity subsolution, supersolution or solution of \eqref{e:aep}, respectively.
We now define the \emph{upper critical value} and \emph{lower critical value} $c^\pm \in \mathbf{R}\cup\{ \pm \infty \}$ by
\begin{align*}
c^+ = c^+(H) &:= \inf\{ a \in \mathbf{R} \mid \text{$a$ is a subeigenvalue of \eqref{e:aep}} \}, \\
c^- = c^-(H) &:= \sup\{ a \in \mathbf{R} \mid \text{$a$ is a supereigenvalue of \eqref{e:aep}} \}.
\end{align*}

For later convenience we prepare several notations:
Let $B(x, r)$ denote the open ball with center $x$ and radius $r > 0$
and let $\overline{B}(x, r)$ denote its closure.
For the graphs $G = \nabla u, \partial u, D^\pm u \subset \mathbf{T}^N\times\mathbf{R}^N$ and a point $x \in \mathbf{T}^N$,
set $G(x) := \{ p \in \mathbf{R}^N \mid (x, p) \in G \}$.
A modulus is a non-negative function $\omega$ defined on $[0, \infty)$ with $\lim_{r \to 0}\omega(r) = 0$.

The following propositions give basic properties of the critical values.

\begin{proposition}[Characterization and rough estimates]
\label{t:cvest}
\begin{equation}
\label{e:cvest1}
c^+(H) = \inf_{u \in \Lip(\mathbf{T}^N)}\sup_{D^+ u}H,
\quad c^-(H) = \sup_{u \in \Lip(\mathbf{T}^N)}\inf_{D^- u}H,
\end{equation}
\begin{equation}
\label{e:cvest2}
\min_{x \in \mathbf{T}^N}H(x, 0)
\le c^\pm(H)
\le \max_{x \in \mathbf{T}^N}H(x, 0).
\end{equation}
\end{proposition}

\begin{proof}
We only show the equation and inequalities for the upper critical value $c^+(H)$
since a symmetric argument shows a proof for the lower critical value $c^-(H)$.
The proof is not so difficult;
for a subeigenvalue $a \in \mathbf{R}$, since there exists a Lipschitz subsolution, $\inf_{u \in \Lip(\mathbf{T}^N)}\sup_{D^+ u}H \le a$.
We also see that Lipschitz functions $u \in \Lip(\mathbf{T}^N)$ themselves are a subsolution of the equation \eqref{e:aep} with $a = \sup_{D^+ u}H$.
Therefore, \eqref{e:cvest1} holds.
Moreover, $u = 0$ is a subsolution of \eqref{e:aep} with $a = \max_{x \in \mathbf{T}^N}H(x, 0)$.
For a subeigenvalue $a \in \mathbf{R}$ and the subsolution $u \in \Lip(\mathbf{T}^N)$ of \eqref{e:aep},
since $(x, 0) \in D^+ u$ at a maximum point $x \in \mathbf{T}^N$ of $u$,
we have $\min_{x \in \mathbf{T}^N}H(x, 0) \le a$.
We have shown \eqref{e:cvest2}.
\end{proof}

\begin{proposition}[Monotonicity of critical values]
\label{t:monocri}
Let $H_1$ and $H_2$ be two Hamiltonians such that $H_1 \le H_2$ on $\mathbf{T}^N\times\mathbf{R}^N$.
Then, $c^\pm(H_1) \le c^\pm(H_2)$, respectively.
\end{proposition}

The proof is trivial so we omit it.

\begin{proposition}[Upper and lower critical values]
\label{t:cvcoincidence}
Assume that $H$ satisfies (A1).
Then, $c^-(H) \le c^+(H)$.
Moreover, if
\item[(A3)]
(Coercivity)
$H$ is coercive in the variable $p \in \mathbf{R}^N$ uniformly in $x \in \mathbf{T}^N$,
i.e.\
$$
\liminf_{|p| \to \infty}\inf_{x \in \mathbf{T}^N}H(x, p) = +\infty,
$$
then $c^-(H) = c^+(H)$ and they are a unique eigenvalue of \eqref{e:aep}.
\end{proposition}

This is a well-known fact;
we refer the reader to \cite{LPV87}, \cite{FS05} and \cite{HNN14}.
Under the assumptions (A1) and (A3) the unique eigenvalue $c = c(H) := c^+(H) = c^-(H)$ is called \emph{critical value} of \eqref{e:aep}.

The generalized effective Hamiltonian introduced in the author's previous work \cite{HNN14} is nothing but the upper critical value $c^+$:

\begin{proposition}
\label{t:cvformula2}
Assume that $H$ satisfies (A1)
and let $H_n$ be a sequence of Hamiltonians satisfying (A1) and (A3).
If $H_n$ converges to $H$ in the sense of
$$\text{
$\liminf_{n}\inf_{\mathbf{T}^N\times\mathbf{R}^N}(H_n-H) \ge 0$,
$\limsup_{n}\sup_{\mathbf{T}^N\times B(0, R)}(H_n-H) \le 0$ for all $R > 0$,
}$$
then $c(H_n) \to c^+(H)$.
\end{proposition}

\begin{proof}
Consider the specific approximation $H_n(x, p) = H(x, p)+|p|/n$ for $n = 1, \cdots$.
Since $H_n \ge H$,
we see by Proposition \ref{t:monocri} that $c(H_n) = c^+(H_n) \ge c^+(H)$,
which immediately yields $\liminf_n c(H_n) \ge c^+(H)$.
In order to the opposite inequality, fix a subeigenvalue $a$ and take the Lipschitz continuous subsolution $u$ of \eqref{e:aep}.
Note that the closure of $D^+ u$ is compact by the Lipschitz continuity.
Hence, $H_n$ becomes coincident to $H$ on $D^+ u$ for sufficiently large $n$.
Therefore, $c(H_n) = c^+(H_n) \le a$, which shows $\limsup_n c(H_n) \le c^+(H)$.
For general approximations one can show by the same arguments as in \cite[Theorem 4.1]{HNN14} that $c(H_n)$ is a convergent sequence and that the limit does not depend on the choice of the approximations.
Finally, we have $\lim_n c(H_n) = c^+(H)$.
\end{proof}

We state our main result.

\begin{theorem}[Minimax formulas]
\label{t:cvformula}
Assume (A1) and (A2') (not (A2)).
Then,
$$
c^+(H)
= \inf_{u \in \Lip(\mathbf{T}^N)}\sup_{\nabla u}H
= \inf_{u \in C^\infty(\mathbf{T}^N)}\sup_{\nabla u}H.
$$
In particular, if (A3) holds,
then they are nothing but the critical value $c(H)$ (the unique eigenvalue of \eqref{e:aep}).
\end{theorem}

Some inequalities hold unconditionally.

\begin{proposition}
$$
\inf_{u \in \Lip(\mathbf{T}^N)}\sup_{\nabla u}H
\le \inf_{u \in \Lip(\mathbf{T}^N)}\sup_{D^+ u}H
= c^+(H)
\le \inf_{u \in \Lip(\mathbf{T}^N)}\sup_{\partial u}H
\le \inf_{u \in C^\infty(\mathbf{T}^N)}\sup_{\nabla u}H.
$$
\end{proposition}

\begin{proof}
These inequalities follow from the well-known orders $\nabla u \subset D^+ u \subset \partial u$ for $u \in \Lip(\mathbf{T}^N)$ and $\nabla u = D^+ u = \partial u$ for $u \in C^1(\mathbf{T}^N)$;
see \cite[Lemma II.1.8 and Subsection II.4.1]{BC97}.
\end{proof}

\section{Proof with fundamental inequality for quasiconvex functions}
\label{s:pfmeasqconv}

Lemma \ref{t:measqconv} will result in the fundamental property of convex sets with probability measures by the level-set convexity of $f$:

\begin{lemma}[Fundamental inclusion for convex sets]
\label{t:measconvset}
Let $C$ be a closed subset of $\mathbf{R}^N$.
Then, $C$ is convex
if and only if
$$
e := \int_\Omega X d\mu \in C
$$
for all measure spaces $(\Omega, \mu)$ with $\mu(\Omega) = 1$ and all $\mathbf{R}^N$-valued integrable functions $X$ on $\Omega$ satisfying $X \in C$ $\mu$-a.e.\ on $\Omega$.
\end{lemma}

\begin{proof}
The ``if'' part is easy;
for $x, y \in C$ and $0 \le \theta \le 1$,
set $\Omega = \{ \pm 1 \}$, $\nu(\{ -1 \}) = \theta$, $\nu(\{ 1 \}) = 1-\theta$, $X(-1) = x$, $X(1) = y$.
Then, since $\int_\Omega X d\mu = \theta x+(1-\theta)y$, we have $\theta x+(1-\theta)y \in C$.

We show the ``only if'' part.
Suppose conversely that $e \notin C$.
Then, by the hyperplane separation theorem (see, e.g., \cite[Theorem 11.4]{R70}) one is able to find a vector $v \in \mathbf{R}^N$ such that
$$
v\cdot x \le a < v\cdot e
\quad \text{for all $x \in C$}
$$
with some $a \in \mathbf{R}$.
Since $X \in C$ a.s.,
$$
v\cdot e
= v\cdot\int_\Omega X d\mu
= \int_\Omega v\cdot X d\mu
\le \int_\Omega a d\mu
= a,
$$
which is contradicts to $v\cdot e > a$.
Therefore, $e \in C$.
\end{proof}

\begin{proof}[Proof of Lemma \ref{t:measqconv}]
The ``if'' part is easy as Lemma \ref{t:measconvset}.
We show the ``only if'' part.
First note that we may assume that $\esssup f\circ X = \sup f\circ X$ since $\tilde{\Omega} := \{ f\circ X \le \esssup_\Omega f\circ X \}$ satisfies $\esssup_{\Omega} f\circ X = \sup_{\tilde{\Omega}} f\circ X$ and $\mu(\Omega\setminus\tilde{\Omega}) = 0$.
Set $E := X(\tilde{\Omega})$ and take its closed convex hull $\cco E$.
Then, Lemma \ref{t:measconvset} shows that $\int_{\tilde{\Omega}} X d\mu \in \cco E$
and therefore
$$
f\left(\int_{\Omega} X d\mu\right)
= f\left(\int_{\tilde{\Omega}} X d\mu\right)
\le \sup_{\cco E}f
= \sup_{E}f
= \sup_{\tilde{\Omega}} f\circ X
= \esssup_\Omega f\circ X
$$
Here, the middle equation follows from the quasiconvexity assumption of $f$.
\end{proof}

\begin{remark}
We can easily prove the standard Jensen's inequality by applying Lemma \ref{t:measconvset} to the closed convex set $\{ (x, y) \mid y \ge f(x) \}$ for a convex function $f$.
\end{remark}

We are now able to show Theorem \ref{t:cvformula}.

\begin{proof}[Proof of Theorem \ref{t:cvformula} using Lemma \ref{t:measqconv}]
It is enough to show
\begin{equation}
\label{e:cvformula3}
\inf_{u \in C^\infty(\mathbf{T}^N)}\sup_{\nabla u}H
\le \inf_{u \in \Lip(\mathbf{T}^N)}\sup_{\nabla u}H.
\end{equation}
Fix $u \in \Lip(\mathbf{T}^N)$ and take the standard mollifications $u_n := u*\eta_n \in C^\infty(\mathbf{T}^N)$.
Note that $H$ is uniformly continuous on $\mathbf{T}^N\times \overline{B}(0, R)$ with $R := \esssup_{\mathbf{T}^N}|\nabla u|$;
there is a modulus $\omega$ such that $|H(x, p)-H(y, q)| \le \omega(|x-y|+|p-q|)$ for all $x, y \in \mathbf{T}^N$ and $p, q \in \overline{B}(0, R)$.
Fix $(x, p) \in \nabla u_n$.
Then, we can calculate that
$$
\begin{aligned}
H(x, p)
&= H(x, \nabla u_n(x))
= H\left(x, \int_{\mathbf{T}^N} \nabla u(x-y)\eta_n(y)d y\right) \\
&\le \esssup_{y \in \spt(\eta_n)}H(x, \nabla u(x-y)) \\
&\le \sup_{\nabla u}H+\esssup_{y \in \spt(\eta_n)}\omega(|y|).
\end{aligned}
$$
Here, we have used the quasiconvexity (A2') and Lemma \ref{t:measqconv} in order to obtain the first inequality.
Taking a limit with respect to $n$,
we have $\sup_{\nabla u_n}H(x, p) \le \sup_{\nabla u}H$,
which implies \eqref{e:cvformula3}.
We have obtained all inequalities to show Theorem \ref{t:cvformula}.
\end{proof}

This proof also shows approximation of viscosity solutions,
whose convex versions have been established in \cite{BJ90} and \cite[Section II.5]{BC97}.

\begin{proposition}[Approximation of viscosity solutions]
\label{t:dstab}
Assume (A1) and (A2').
Let $u \in \Lip(\mathbf{T}^N)$ and let $u_n \in C^\infty(\mathbf{T}^N)$ be the standard mollification $u*\eta_n$.
If $u$ is a viscosity subsolution of \eqref{e:aep},
then $u_n$ are a viscosity subsolution of $H(x, D u_n) = a+\omega(1/n)$ in $\mathbf{T}^N$ with some modulus $\omega$.
\end{proposition}

\section{Proof with generalized gradients}
\label{s:pfmollip}

We begin with:

\begin{proposition}
\label{t:cvest2}
Assume (A1) and (A2').
Then,
$$
\inf_{u \in \Lip(\mathbf{T}^N)}\sup_{\nabla u}H
= \inf_{u \in \Lip(\mathbf{T}^N)}\sup_{D^+ u}H
= \inf_{u \in \Lip(\mathbf{T}^N)}\sup_{\partial u}H.
$$
\end{proposition}

\begin{proof}
This is true since
\begin{equation}
\label{e:vsolgsol}
\sup_{\nabla u}H
= \sup_{\overline{\nabla u}}H
= \sup_{\cco_p\overline{\nabla u}}H
= \sup_{\partial u}H
\end{equation}
for all $u \in \Lip(\mathbf{T}^N)$.
In order to obtain the second equality,
we need the quasiconvexity assumption (A2').
\end{proof}

\begin{remark}
This proof also implies \cite[Lemma 2.1]{AS13}.
\end{remark}

We prove Lemma \ref{t:mollip}
in order to show the remaining inequality in Theorem \ref{t:cvformula}
\begin{equation}
\label{e:cvformula4}
\inf_{u \in C^\infty(\mathbf{T}^N)}\sup_{\nabla u}H
\le \inf_{u \in \Lip(\mathbf{T}^N)}\sup_{\partial u}H.
\end{equation}
The proof, which uses Jensen's inequality to distance functions from convex sets, is due to A.~Siconolfi.
A similar technique appears in \cite{FS05}.
We first prepare:

\begin{lemma}[Continuity of generalized gradients]
\label{t:distlip}
Let $u \in \Lip(\mathbf{T}^N)$.
Then, for each $x \in \mathbf{T}^N$ there exists a modulus $\omega_x$ such that
\begin{equation}
\label{e:distlip}
d(\partial u(x), p) \le \omega_x(|y-x|)
\quad \text{for all $(y, p) \in \partial u$.}
\end{equation}
\end{lemma}

This lemma means that the the generalized gradients $\partial u$ is upper semicontinuous as a set-valued function.
The proof is easy since $\partial u$ is compact (see, e.g., \cite[Proposition 2.1.5]{C83} and \cite[Proposition 1.4.8]{AF90})
but we prove it for completeness.

\begin{proof}
Fix arbitrary $\varepsilon > 0$.
Since $\partial u$ and $\{ x \}\times\{ p \mid d(\partial u(x), p) = \varepsilon \}$ are disjoint compact sets,
$\partial u$ and $B(x, \delta)\times\{ p \mid d(\partial u(x), p) \ge \varepsilon-\delta \}$ have empty intersections for some small $\delta > 0$.
Therefore, every $(y, p) \in \partial u$ with $|y-x| < \delta$ satisfies $d(\partial u(x), p) < \varepsilon-\delta < \varepsilon$.
\end{proof}

\begin{proof}[Proof of Lemma \ref{t:mollip}]
First note that the set $\partial u(x)$ is non-empty closed convex and hence $d(\partial u(x), \cdot)$ is a (Lipschitz) continuous convex function on $\mathbf{R}^N$.
We observe by Jensen's inequality that
$$
\begin{aligned}
d(\partial u(x), q)
&= d(\partial u(x), \nabla u_n(y))
= d\left(\partial u(x), \int_{\mathbf{T}^N} \nabla u(y-z)\eta_n(z)d z\right) \\
&\le \int_{\mathbf{T}^N} d(\partial u(x), \nabla u(y-z))\eta_n(z)d z
\end{aligned}
$$
for all $(y, q) \in \nabla u_n$.
By Lemma \ref{t:distlip} we have
$$
\begin{aligned}
d(\partial u(x), p_n)
&\le \int_{\mathbf{T}^N} d(\partial u(x), \nabla u(x_n-z))\eta_n(z)d z \\
&\le \int_{\mathbf{T}^N} \omega_x(|x_n-z-x|)\eta_n(z)d z
\le \sup_{z \in \spt \eta_n} \omega_x(|x_n-z-x|).
\end{aligned}
$$
This shows that $d(\partial u(x), p_n) \to 0$
and therefore $p \in \partial u(x)$.
\end{proof}

Lemma \ref{t:mollip} yields another proof of Theorem \ref{t:cvformula}.

\begin{proof}[Proof of Theorem \ref{t:cvformula} using Lemma \ref{t:mollip}]
It is enough to show \eqref{e:cvformula4}.
Fix $u \in \Lip(\mathbf{T}^N)$ and take the standard mollifications $u_n := u*\eta_n \in C^\infty(\mathbf{T}^N)$.
Also take a maximum point $(x_n, p_n) \in \nabla u_n$ of $H$ so that $H(x_n, p_n) = \sup_{\nabla u_n}H$.
Now, note that the sequence $(x_n, p_n)$ has an accumulation point $(x, p) \in \mathbf{T}^N\times \overline{B}(0, \esssup_{\mathbf{T}^N}|\nabla u|)$ since $u$ is Lipschitz continuous.
We then see by Lemma \ref{t:mollip} that $(x, p) \in \partial u$
and therefore
$$
\inf_{u \in C^\infty(\mathbf{T}^N)}\sup_{\nabla u}H
\le \sup_{\nabla u_n}H
= H(x_n, p_n)
\to H(x, p)
\le \sup_{\partial u}H.
$$
Since $u \in \Lip(\mathbf{T}^N)$ is arbitrary, we can obtain the desired inequality \eqref{e:cvformula4}.
We now have obtained all the equations in Theorem \ref{t:cvformula}.
\end{proof}

\begin{remark}
This proof is a bit longer than the proof in Section \ref{s:pfmeasqconv}
but may give a deeper observation.
For example, there is a question that if $u_n \in C^\infty$ converges to $u \in \Lip$ uniformly, then a sequence $(x_n, p_n) \in \nabla u_n$ has an accumulation point belonging to $\partial u$.
This is an open problem concerned with stability of viscosity solutions.
We also remark that one is able to prove Proposition \ref{t:dstab} by combining Lemma \ref{t:mollip} and the equation \eqref{e:vsolgsol}.
\end{remark}

\section*{Acknowledgments}

The author would like to thank Antonio~Siconolfi for insightful comments and fruitful discussions on this problem.
His suggestion for Lemma \ref{t:mollip} is crucial in this work.
The author is also grateful to Yoshikazu~Giga, Hiroyoshi~Mitake, Hung~Tran, Nao~Hamamuki and Tokinaga~Namba for their useful remarks.
In particular, Mitake and Tran provided many suggestions to improve this article.
This work started from Tran's intensive lectures on stochastic homogenization in Tokyo.

\providecommand{\bysame}{\leavevmode\hbox to3em{\hrulefill}\thinspace}
\providecommand{\MR}{\relax\ifhmode\unskip\space\fi MR }
\providecommand{\MRhref}[2]{  \href{http://www.ams.org/mathscinet-getitem?mr=#1}{#2}
}
\providecommand{\href}[2]{#2}

\end{document}